\documentclass{amsart}
\usepackage{amssymb}
\usepackage{amsfonts}

\newtheorem{theorem}{Theorem}
\theoremstyle{plain}

\newtheorem{corollary}{Corollary}

\newtheorem{definition}{Definition}

\newtheorem{lemma}{Lemma}

\numberwithin{equation}{section}
\input{tcilatex}

\begin{document}
\title[Reverses and Refinements of Jensen's Inequality]{Reverses and
Refinements of Jensen's Inequality for Positive Linear Functionals on
Hermitian Unital Banach $\ast $-Algebras }
\author{S. S. Dragomir$^{1,2}$}
\address{$^{1}$Mathematics, College of Engineering \& Science\\
Victoria University, PO Box 14428\\
Melbourne City, MC 8001, Australia.}
\email{sever.dragomir@vu.edu.au}
\urladdr{http://rgmia.org/dragomir}
\address{$^{2}$DST-NRF Centre of Excellence \\
in the Mathematical and Statistical Sciences, School of Computer Science \&
Applied Mathematics, University of the Witwatersrand, Private Bag 3,
Johannesburg 2050, South Africa}

\begin{abstract}
We establish in this paper some inequalities for analytic and convex
functions on an open interval and positive normalized functionals defined on
a Hermitian unital Banach $\ast $-algebra. Reverses and refinements of
Jensen's and Slater's type inequalities are provided. Some examples for
particular convex functions of interest are given as well.
\end{abstract}

\maketitle

\section{Introduction}

We need some preliminary concepts and facts about Banach $\ast $-algebras.

Let $A$ be a unital Banach $\ast $-algebra with unit $1$. An element $a\in A$
is called \textit{selfadjoint} if $a^{\ast }=a.$ $A$ is called \textit{%
Hermitian} if every selfadjoint element $a$ in $A$ has real \textit{spectrum}
$\sigma \left( a\right) ,$ namely $\sigma \left( a\right) \subset \mathbb{R}$%
.

We say that an element $a$ is \textit{nonnegative} and write this as $a\geq
0 $ if $a^{\ast }=a$ and $\sigma \left( a\right) \subset \left[ 0,\infty
\right) .$ We say that $a$ is \textit{positive }and write $a>0$ if $a\geq 0$
and $0\notin \sigma \left( a\right) .$ Thus $a>0$ implies that its inverse $%
a^{-1}$ exists. Denote the set of all invertible elements of $A$ by $%
\limfunc{Inv}\left( A\right) .$ If $a,$ $b\in \limfunc{Inv}\left( A\right) ,$
then $ab\in \limfunc{Inv}\left( A\right) $ and $\left( ab\right)
^{-1}=b^{-1}a^{-1}.$ Also, saying that $a\geq b$ means that $a-b\geq 0$ and,
similarly $a>b$ means that $a-b>0.$

The \textit{Shirali-Ford theorem} asserts that if $A$ is a unital Banach $%
\ast $-algebra \cite{SF} (see also \cite[Theorem 41.5]{BD}), then%
\begin{equation}
\left\vert a\right\vert ^{2}:=a^{\ast }a\geq 0\text{ for every }a\in A. 
\tag{SF}  \label{SF}
\end{equation}%
Based on this fact, Okayasu \cite{O}, Tanahashi and Uchiyama \cite{TU}
proved the following fundamental properties (see also \cite{F}):

\begin{enumerate}
\item[(i)] If $a,$ $b\in A,$ then $a\geq 0,$ $b\geq 0$ imply $a+b\geq 0$ and 
$\alpha \geq 0$ implies $\alpha a\geq 0;$

\item[(ii)] If $a,$ $b\in A,$ then $a>0,$ $b\geq 0$ imply $a+b>0;$

\item[(iii)] If $a,$ $b\in A,$ then either $a\geq b>0$ or $a>b\geq 0$ imply $%
a>0;$

\item[(iv)] If $a>0,$ then $a^{-1}>0;$

\item[(v)] If $c>0,$ then $0<b<a$ if and only if $cbc<cac,$ also $0<b\leq a$
if and only if $cbc\leq cac;$

\item[(vi)] If $0<a<1,$ then $1<a^{-1};$

\item[(vii)] If $0<b<a,$ then $0<a^{-1}<b^{-1},$ also if $0<b\leq a,$ then $%
0<a^{-1}\leq b^{-1}.$
\end{enumerate}

In order to introduce the real power of a positive element, we need the
following facts \cite[Theorem 41.5]{BD}.

Let $a\in A$ and $a>0,$ then $0\notin \sigma \left( a\right) $ and the fact
that $\sigma \left( a\right) $ is a compact subset of $\mathbb{C}$ implies
that $\inf \{z:z\in \sigma \left( a\right) \}>0$ and $\sup \{z:z\in \sigma
\left( a\right) \}<\infty .$ Choose $\gamma $ to be close rectifiable curve
in $\{\func{Re}z>0\},$ the right half open plane of the complex plane, such
that $\sigma \left( a\right) \subset \limfunc{ins}\left( \gamma \right) ,$
the inside of $\gamma .$ Let $G$ be an open subset of $\mathbb{C}$ with $%
\sigma \left( a\right) \subset G.$ If $f:G\rightarrow \mathbb{C}$ is
analytic, we define an element $f\left( a\right) $ in $A$ by 
\begin{equation*}
f\left( a\right) :=\frac{1}{2\pi i}\int_{\gamma }f\left( z\right) \left(
z-a\right) ^{-1}dz.
\end{equation*}%
It is well known (see for instance \cite[pp. 201-204]{C}) that $f\left(
a\right) $ does not depend on the choice of $\gamma $ and the Spectral
Mapping Theorem (SMT) 
\begin{equation*}
\sigma \left( f\left( a\right) \right) =f\left( \sigma \left( a\right)
\right)
\end{equation*}
holds.

For any $\alpha \in \mathbb{R}$ we define for $a\in A$ and $a>0,$ the real
power 
\begin{equation*}
a^{\alpha }:=\frac{1}{2\pi i}\int_{\gamma }z^{\alpha }\left( z-a\right)
^{-1}dz,
\end{equation*}%
where $z^{\alpha }$ is the principal $\alpha $-power of $z.$ Since $A$ is a
Banach $\ast $-algebra, then $a^{\alpha }\in A.$ Moreover, since $z^{\alpha
} $ is analytic in $\{\func{Re}z>0\},$ then by (SMT) we have%
\begin{equation*}
\sigma \left( a^{\alpha }\right) =\left( \sigma \left( a\right) \right)
^{\alpha }=\{z^{\alpha }:z\in \sigma \left( a\right) \}\subset \left(
0,\infty \right) .
\end{equation*}

Following \cite{F}, we list below some important properties of real powers:

\begin{enumerate}
\item[(viii)] If $0<a\in A$ and $\alpha \in \mathbb{R}$, then $a^{\alpha
}\in A$ with $a^{\alpha }>0$ and $\left( a^{2}\right) ^{1/2}=a,$ \cite[Lemma
6]{TU};

\item[(ix)] If $0<a\in A$ and $\alpha ,$ $\beta \in \mathbb{R}$, then $%
a^{\alpha }a^{\beta }=a^{\alpha +\beta };$

\item[(x)] If $0<a\in A$ and $\alpha \in \mathbb{R}$, then $\left( a^{\alpha
}\right) ^{-1}=\left( a^{-1}\right) ^{\alpha }=a^{-\alpha };$

\item[(xi)] If $0<a,$ $b\in A$, $\alpha ,$ $\beta \in \mathbb{R}$ and $%
ab=ba, $ then $a^{\alpha }b^{\beta }=b^{\beta }a^{\alpha }.$
\end{enumerate}

Okayasu \cite{O} showed that the \textit{L\"{o}wner-Heinz inequality}
remains valid in a Hermitian unital Banach $\ast $-algebra with continuous
involution, namely if $a,$ $b\in A$ and $p\in \left[ 0,1\right] $ then $a>b$ 
$\left( a\geq b\right) $ implies that $a^{p}>b^{p}$ $\left( a^{p}\geq
b^{p}\right) .$

Now, assume that $f\left( \cdot \right) $ is analytic in $G$, an open subset
of $\mathbb{C}$ and for the real interval $I\subset G$ assume that $f\left(
z\right) \geq 0$ for any $z\in I.$ If $u\in A$ such that $\sigma \left(
u\right) \subset I,$ then by (SMT) we have%
\begin{equation*}
\sigma \left( f\left( u\right) \right) =f\left( \sigma \left( u\right)
\right) \subset f\left( I\right) \subset \left[ 0,\infty \right)
\end{equation*}%
meaning that $f\left( u\right) \geq 0$ in the order of $A.$

Therefore, we can state the following fact that will be used to establish
various inequalities in $A,$ see also \cite{SSDQM}.

\begin{lemma}
\label{l.1.2}Let $f\left( z\right) $ and $g\left( z\right) $ be analytic in $%
G$, an open subset of $\mathbb{C}$ and for the real interval $I\subset G,$
assume that $f\left( z\right) \geq g\left( z\right) $ for any $z\in I.$ Then
for any $u\in A$ with $\sigma \left( u\right) \subset I$ we have $f\left(
u\right) \geq g\left( u\right) $ in the order of $A.$
\end{lemma}

\begin{definition}
\label{d.2.1}Assume that $A$ is a Hermitian unital Banach $\ast $-algebra. A
linear functional $\psi :A\rightarrow \mathbb{C}$ is positive if for $a\geq
0 $ we have $\psi \left( a\right) \geq 0.$ We say that it is normalized if $%
\psi \left( 1\right) =1.$
\end{definition}

We observe that the positive linear functional $\psi $ preserves the order
relation, namely if $a\geq b$ then $\psi \left( a\right) \geq \psi \left(
b\right) $ and if $\beta \geq a\geq \alpha $ with $\alpha ,$ $\beta $ real
numbers, then $\beta \geq \psi \left( a\right) \geq \alpha .$

In the recent paper \cite{SSDMcC}\ we established the following McCarthy
type inequality:

\begin{theorem}
\label{t.A}Assume that $A$ is a Hermitian unital Banach $\ast $-algebra and $%
\psi :A\rightarrow \mathbb{C}$ a positive normalized linear functional on $%
A. $

(i) If $p\in \left( 0,1\right) $ and $a\geq 0,$ then%
\begin{equation}
\psi ^{p}\left( a\right) \geq \psi \left( a^{p}\right) \geq 0;  \label{e.1.1}
\end{equation}

(ii) If $q\geq 1$ and $b\geq 0,$ then 
\begin{equation}
\psi \left( b^{q}\right) \geq \psi ^{q}\left( b\right) \geq 0;  \label{e.1.2}
\end{equation}%
(iii) If $r<0,$ $c>0$ with $\psi \left( c\right) >0,$ then 
\begin{equation}
\psi \left( c^{r}\right) \geq \psi ^{r}\left( c\right) >0.  \label{e.1.3}
\end{equation}
\end{theorem}

In \cite{SSDJen} and \cite{SSDGru} we obtained the following result for
analytic convex functions:

\begin{theorem}
\label{t.B}Let $f\left( z\right) $ be analytic in $G$, an open subset of $%
\mathbb{C}$ and the real interval $I\subset G.$ If $f$ is convex (in the
usual sense) on the interval $I$ and $\psi :A\rightarrow \mathbb{C}$ is a
positive normalized linear functional on $A,$ then for any selfadjoint
element $c\in A$ with with $\sigma \left( c\right) \subseteq \left[ m,M%
\right] \subset I$ for some real numbers $m<M,$ 
\begin{align}
0& \leq \psi \left( f\left( c\right) \right) -f\left( \psi \left( c\right)
\right) \leq \psi \left( cf^{\prime }\left( c\right) \right) -\psi \left(
c\right) \psi \left( f^{\prime }\left( c\right) \right)  \label{e.1.4} \\
&  \notag \\
& \leq \left\{ 
\begin{array}{l}
\frac{1}{2}\left( M-m\right) \left[ \psi \left( \left[ f^{\prime }\left(
c\right) \right] ^{2}\right) -\psi ^{2}\left( f^{\prime }\left( c\right)
\right) \right] ^{1/2} \\ 
\\ 
\frac{1}{2}\left[ f^{\prime }\left( M\right) -f^{\prime }\left( m\right) %
\right] \left( \psi \left( c^{2}\right) -\psi ^{2}\left( c\right) \right)
^{1/2}%
\end{array}%
\right.  \notag \\
&  \notag \\
& \leq \frac{1}{4}\left( M-m\right) \left[ f^{\prime }\left( M\right)
-f^{\prime }\left( m\right) \right] .  \notag
\end{align}
\end{theorem}

Motivated by these results we establish in this paper some inequalities for
analytic and convex functions on an open interval and positive normalized
functionals defined on a Hermitian unital Banach $\ast $-algebra. Reverses
and refinements of Jensen's and Slater's type inequalities are provided.
Some examples for particular convex functions of interest are given as well.

\section{Some Reverses}

We have:

\begin{theorem}
\label{t.2.1}Let $f\left( z\right) $ be analytic in $G$, an open subset of $%
\mathbb{C}$ and the real interval $I\subset G.$ If $f$ is convex on the
interval $I$ and $\psi :A\rightarrow \mathbb{C}$ is a positive normalized
linear functional on $A,$ then for any selfadjoint element $c\in A$ with $%
\sigma \left( c\right) \subseteq \left[ m,M\right] \subset I$ for some real
numbers $m<M,$ 
\begin{align}
0& \leq \psi \left( f\left( c\right) \right) -f\left( \psi \left( c\right)
\right)   \label{e.2.3} \\
&  \notag \\
& \leq \frac{\left( M-\psi \left( c\right) \right) \left( \psi \left(
c\right) -m\right) }{M-m}\sup_{t\in \left( m,M\right) }\Theta _{f}\left(
t;m,M\right)   \notag \\
&  \notag \\
& \leq \left\{ 
\begin{array}{c}
\frac{1}{4}\left( M-m\right) \sup_{t\in \left( m,M\right) }\Theta _{f}\left(
t;m,M\right)  \\ 
\\ 
\left( M-\psi \left( c\right) \right) \left( \psi \left( c\right) -m\right) 
\frac{f^{\prime }\left( M\right) -f^{\prime }\left( m\right) }{M-m}%
\end{array}%
\right.   \notag \\
&  \notag \\
& \leq \frac{1}{4}\left( M-m\right) \left[ f^{\prime }\left( M\right)
-f^{\prime }\left( m\right) \right]   \notag
\end{align}%
provided $\psi \left( c\right) \in \left( m,M\right) ,$ where $\Theta
_{f}\left( \cdot ;m,M\right) :\left( m,M\right) \rightarrow \mathbb{R}$ is
defined by%
\begin{equation*}
\Theta _{f}\left( t;m,M\right) =\frac{f\left( M\right) -f\left( t\right) }{%
M-t}-\frac{f\left( t\right) -f\left( m\right) }{t-m}.
\end{equation*}

We also have%
\begin{align}
0& \leq \psi \left( f\left( c\right) \right) -f\left( \psi \left( c\right)
\right) \leq \frac{1}{4}\left( M-m\right) \Theta _{f}\left( \psi \left(
c\right) ;m,M\right)  \label{e.2.3.1} \\
&  \notag \\
& \leq \frac{1}{4}\left( M-m\right) \sup_{t\in \left( m,M\right) }\Theta
_{f}\left( t;m,M\right) \leq \frac{1}{4}\left( M-m\right) \left[ f^{\prime
}\left( M\right) -f^{\prime }\left( m\right) \right] ,  \notag
\end{align}%
provided $\psi \left( c\right) \in \left( m,M\right) .$
\end{theorem}

\begin{proof}
By the convexity of $f$ on $\left[ m,M\right] $ we have for any $z\in \left[
m,M\right] $ that%
\begin{equation}
f\left( z\right) \leq \frac{z-m}{M-m}f\left( M\right) +\frac{M-z}{M-m}%
f\left( m\right) .  \label{e.2.4}
\end{equation}%
Using Lemma \ref{l.1.2} we have by (\ref{e.2.4}) for any selfadjoint element 
$c\in A$ with $\sigma \left( c\right) \subseteq \left[ m,M\right] $ that%
\begin{equation}
f\left( c\right) \leq f\left( M\right) \frac{c-m}{M-m}+f\left( m\right) 
\frac{M-c}{M-m}  \label{e.2.5}
\end{equation}%
in the order of $A.$

If we take in this inequality the functional $\psi $ we get the following
reverse of Jensen's inequality%
\begin{equation}
\psi \left( f\left( c\right) \right) \leq f\left( M\right) \frac{\psi \left(
c\right) -m}{M-m}+f\left( m\right) \frac{M-\psi \left( c\right) }{M-m}.
\label{e.2.6}
\end{equation}%
This generalizes the scalar Lah-Ribari\'{c} inequality for convex functions
that is well known in the literature, see for instance \cite[p. 57]{FMPS}
for an extension to selfadjoint operators in Hilbert spaces.

Define 
\begin{equation*}
\Delta _{f}\left( t;m,M\right) :=\frac{\left( t-m\right) f\left( M\right)
+\left( M-t\right) f\left( m\right) }{M-m}-f\left( t\right) ,\quad t\in %
\left[ m,M\right] ,
\end{equation*}%
then we have%
\begin{align}
\Delta _{f}\left( t;m,M\right) & =\frac{\left( t-m\right) f\left( M\right)
+\left( M-t\right) f\left( m\right) -\left( M-m\right) f\left( t\right) }{M-m%
}  \label{e.2.7} \\
& =\frac{\left( t-m\right) f\left( M\right) +\left( M-t\right) f\left(
m\right) -\left( M-t+t-m\right) f\left( t\right) }{M-m}  \notag \\
& =\frac{\left( t-m\right) \left[ f\left( M\right) -f\left( t\right) \right]
-\left( M-t\right) \left[ f\left( t\right) -f\left( m\right) \right] }{M-m} 
\notag \\
& =\frac{\left( M-t\right) \left( t-m\right) }{M-m}\Theta _{f}\left(
t;m,M\right)   \notag
\end{align}%
for any $t\in \left( m,M\right) .$

From (\ref{e.2.6}) we have for $\psi \left( c\right) \in \left( m,M\right) $
that 
\begin{align}
& \psi \left( f\left( c\right) \right) -f\left( \psi \left( c\right) \right)
\label{e.2.8} \\
& \leq \frac{\left( \psi \left( c\right) -m\right) f\left( M\right) +\left(
M-\psi \left( c\right) \right) f\left( m\right) }{M-m}-f\left( \psi \left(
c\right) \right)  \notag \\
& =\Delta _{f}\left( \psi \left( c\right) ;m,M\right) =\frac{\left( M-\psi
\left( c\right) \right) \left( \psi \left( c\right) -m\right) }{M-m}\Theta
_{f}\left( \psi \left( c\right) ;m,M\right)  \notag \\
& \leq \frac{\left( M-\psi \left( c\right) \right) \left( \psi \left(
c\right) -m\right) }{M-m}\sup_{t\in \left( m,M\right) }\Theta _{f}\left(
t;m,M\right) .  \notag
\end{align}

We also have%
\begin{align*}
\sup_{t\in \left( m,M\right) }\Theta _{f}\left( t;m,M\right) & =\sup_{t\in
\left( m,M\right) }\left[ \frac{f\left( M\right) -f\left( t\right) }{M-t}-%
\frac{f\left( t\right) -f\left( m\right) }{t-m}\right] \\
& \leq \sup_{t\in \left( m,M\right) }\left[ \frac{f\left( M\right) -f\left(
t\right) }{M-t}\right] +\sup_{t\in \left( m,M\right) }\left[ -\frac{f\left(
t\right) -f\left( m\right) }{t-m}\right] \\
& =\sup_{t\in \left( m,M\right) }\left[ \frac{f\left( M\right) -f\left(
t\right) }{M-t}\right] -\inf_{t\in \left( m,M\right) }\left[ \frac{\Phi
\left( t\right) -\Phi \left( m\right) }{t-m}\right] \\
& =f^{\prime }\left( M\right) -f^{\prime }\left( m\right)
\end{align*}%
and since, obviously%
\begin{equation*}
\frac{\left( M-\psi \left( c\right) \right) \left( \psi \left( c\right)
-m\right) }{M-m}\leq \frac{1}{4}\left( M-m\right)
\end{equation*}%
we have the desired result (\ref{e.2.3}).

From (\ref{e.2.8}) we have 
\begin{align*}
& \psi \left( f\left( c\right) \right) -f\left( \psi \left( c\right) \right)
\leq \frac{\left( M-\psi \left( c\right) \right) \left( \psi \left( c\right)
-m\right) }{M-m}\Theta _{f}\left( \psi \left( c\right) ;m,M\right) \\
& \leq \frac{1}{4}\left( M-m\right) \Theta _{f}\left( \psi \left( c\right)
;m,M\right) \leq \frac{1}{4}\left( M-m\right) \sup_{t\in \left( m,M\right)
}\Theta _{f}\left( t;m,M\right) \\
& \leq \frac{1}{4}\left( M-m\right) \left[ f^{\prime }\left( M\right)
-f^{\prime }\left( m\right) \right]
\end{align*}%
that proves (\ref{e.2.3.1}).
\end{proof}

We also have:

\begin{theorem}
\label{t.2.2}With the assumptions of Theorem \ref{t.2.1} we have%
\begin{align}
0& \leq \psi \left( f\left( c\right) \right) -f\left( \psi \left( c\right)
\right)   \label{e.2.9} \\
&  \notag \\
& \leq \left( 1+2\frac{\left\vert \psi \left( c\right) -\frac{m+M}{2}%
\right\vert }{M-m}\right) \left[ \frac{f\left( m\right) +f\left( M\right) }{2%
}-f\left( \frac{m+M}{2}\right) \right]   \notag \\
&  \notag \\
& \leq f\left( m\right) +f\left( M\right) -2f\left( \frac{m+M}{2}\right) . 
\notag
\end{align}
\end{theorem}

\begin{proof}
First of all, we recall the following result obtained by the author in \cite%
{SSDBul}\ that provides a refinement and a reverse for the weighted Jensen's
discrete inequality:%
\begin{align}
& n\min_{i\in \left\{ 1,...,n\right\} }\left\{ p_{i}\right\} \left[ \frac{1}{%
n}\sum_{i=1}^{n}\Phi \left( x_{i}\right) -\Phi \left( \frac{1}{n}%
\sum_{i=1}^{n}x_{i}\right) \right]  \label{e.2.10} \\
& \leq \frac{1}{P_{n}}\sum_{i=1}^{n}p_{i}\Phi \left( x_{i}\right) -\Phi
\left( \frac{1}{P_{n}}\sum_{i=1}^{n}p_{i}x_{i}\right)  \notag \\
& n\max_{i\in \left\{ 1,...,n\right\} }\left\{ p_{i}\right\} \left[ \frac{1}{%
n}\sum_{i=1}^{n}\Phi \left( x_{i}\right) -\Phi \left( \frac{1}{n}%
\sum_{i=1}^{n}x_{i}\right) \right] ,  \notag
\end{align}%
where $\Phi :C\rightarrow \mathbb{R}$ is a convex function defined on the
convex subset $C$ of the linear space $X,$ $\left\{ x_{i}\right\} _{i\in
\left\{ 1,...,n\right\} }\subset C$ are vectors and $\left\{ p_{i}\right\}
_{i\in \left\{ 1,...,n\right\} }$ are nonnegative numbers with $%
P_{n}:=\sum_{i=1}^{n}p_{i}>0.$

For $n=2$ we deduce from (\ref{e.2.10}) that%
\begin{align}
& 2\min \left\{ t,1-t\right\} \left[ \frac{\Phi \left( x\right) +\Phi \left(
y\right) }{2}-\Phi \left( \frac{x+y}{2}\right) \right]  \label{e.2.11} \\
& \leq t\Phi \left( x\right) +\left( 1-t\right) \Phi \left( y\right) -\Phi
\left( tx+\left( 1-t\right) y\right)  \notag \\
& \leq 2\max \left\{ t,1-t\right\} \left[ \frac{\Phi \left( x\right) +\Phi
\left( y\right) }{2}-\Phi \left( \frac{x+y}{2}\right) \right]  \notag
\end{align}%
for any $x,y\in C$ and $t\in \left[ 0,1\right] .$

If we use the second inequality in (\ref{e.2.11}) for the convex function $%
f:I\rightarrow \mathbb{R}$ and $m,$ $M\in \mathbb{R}$, $m<M$ with $\left[ m,M%
\right] \subset I,$ we have for $t=\frac{M-\psi \left( c\right) }{M-m}$ that 
\begin{align}
& \frac{\left( M-\psi \left( c\right) \right) f\left( m\right) +\left( \psi
\left( c\right) -m\right) f\left( M\right) }{M-m}  \label{e.2.12} \\
& -f\left( \frac{m\left( M-\psi \left( c\right) \right) +M\left( \psi \left(
c\right) -m\right) }{M-m}\right)   \notag \\
& \leq 2\max \left\{ \frac{M-\psi \left( c\right) }{M-m},\frac{\psi \left(
c\right) -m}{M-m}\right\} \left[ \frac{f\left( m\right) +f\left( M\right) }{2%
}-f\left( \frac{m+M}{2}\right) \right] ,  \notag
\end{align}%
namely%
\begin{align}
& \frac{\left( M-\psi \left( c\right) \right) f\left( m\right) +\left( \psi
\left( c\right) -m\right) f\left( M\right) }{M-m}-f\left( \psi \left(
c\right) \right)   \label{e.2.13} \\
& \leq \left( 1+2\frac{\left\vert \psi \left( c\right) -\frac{m+M}{2}%
\right\vert }{M-m}\right) \left[ \frac{f\left( m\right) +f\left( M\right) }{2%
}-f\left( \frac{m+M}{2}\right) \right]   \notag \\
& \times \left[ \frac{f\left( m\right) +f\left( M\right) }{2}-f\left( \frac{%
m+M}{2}\right) \right] .  \notag
\end{align}%
On making use of the first inequality in (\ref{e.2.8}) and (\ref{e.2.13}) we
get the first part of (\ref{e.2.9}).

The last part follows by the fact that $m\leq \psi \left( c\right) \leq M.$
\end{proof}

\section{Refinements and Reverses}

We start with the following result:

\begin{theorem}
\label{t.3.1}Let $f\left( z\right) $ be analytic in $G$, an open subset of $%
\mathbb{C}$ and the real interval $I\subset G,$ $\left[ m,M\right] \subset I$
for some real numbers $m<M,$ and $\psi :A\rightarrow \mathbb{C}$ is a
positive normalized linear functional on $A.$ If there exists the constants $%
K>k\geq 0$ such that 
\begin{equation}
K\geq f^{\prime \prime }\left( z\right) \geq k\text{ for any }z\in \left[ m,M%
\right] ,  \label{kK}
\end{equation}%
then for any selfadjoint element $c\in A$ with $\sigma \left( c\right)
\subseteq \left[ m,M\right] \subset I,$%
\begin{equation}
\frac{1}{2}K\psi \left[ \left( c-t\right) ^{2}\right] \geq \psi \left(
f\left( c\right) \right) -f^{\prime }\left( t\right) \left( \psi \left(
c\right) -t\right) -f\left( t\right) \geq \frac{1}{2}k\psi \left[ \left(
c-t\right) ^{2}\right]   \label{e.3.1}
\end{equation}%
and 
\begin{equation}
\frac{1}{2}K\psi \left[ \left( c-t\right) ^{2}\right] \geq \psi \left(
cf^{\prime }\left( c\right) \right) -t\psi \left( f^{\prime }\left( c\right)
\right) +f\left( t\right) -\psi \left( f\left( c\right) \right) \geq \frac{1%
}{2}k\psi \left[ \left( c-t\right) ^{2}\right] ,  \label{e.3.2}
\end{equation}%
for any $t\in \left[ m,M\right] .$
\end{theorem}

\begin{proof}
Using Taylor's representation with the integral remainder we can write the
following identity 
\begin{equation}
f\left( z\right) =\sum_{k=0}^{n}\frac{1}{k!}f^{\left( k\right) }\left(
t\right) \left( z-t\right) ^{k}+\frac{1}{n!}\int_{t}^{z}f^{\left( n+1\right)
}\left( s\right) \left( z-s\right) ^{n}ds  \label{e.3.3}
\end{equation}%
for any $z,$ $t\in \mathring{I},$ the interior of $I.$

For any integrable function $h$ on an interval and any distinct numbers $c,$ 
$d$ in that interval, we have, by the change of variable $s=\left(
1-s\right) c+sd,$ $s\in \left[ 0,1\right] $ that%
\begin{equation*}
\int_{c}^{d}h\left( s\right) ds=\left( d-c\right) \int_{0}^{1}h\left( \left(
1-s\right) c+sd\right) ds.
\end{equation*}%
Therefore,%
\begin{align*}
& \int_{t}^{z}f^{\left( n+1\right) }\left( s\right) \left( z-s\right) ^{n}ds
\\
& =\left( z-t\right) \int_{0}^{1}f^{\left( n+1\right) }\left( \left(
1-s\right) t+sz\right) \left( z-\left( 1-s\right) t-sz\right) ^{n}ds \\
& =\left( z-t\right) ^{n+1}\int_{0}^{1}f^{\left( n+1\right) }\left( \left(
1-s\right) t+sz\right) \left( 1-s\right) ^{n}ds.
\end{align*}%
The identity (\ref{e.3.3}) can then be written as 
\begin{align}
f\left( z\right) & =\sum_{k=0}^{n}\frac{1}{k!}f^{\left( k\right) }\left(
t\right) \left( z-t\right) ^{k}  \label{e.3.4} \\
& +\frac{1}{n!}\left( z-t\right) ^{n+1}\int_{0}^{1}f^{\left( n+1\right)
}\left( \left( 1-s\right) t+sz\right) \left( 1-s\right) ^{n}ds.  \notag
\end{align}%
For $n=1$ we get%
\begin{equation}
f\left( z\right) =f\left( t\right) +\left( z-t\right) f^{\prime }\left(
t\right) +\left( z-t\right) ^{2}\int_{0}^{1}f^{\prime \prime }\left( \left(
1-s\right) t+sz\right) \left( 1-s\right) ds  \label{e.3.5}
\end{equation}%
for any $z,$ $t\in \mathring{I}.$

By the condition (\ref{kK}) we have%
\begin{equation*}
K\int_{0}^{1}\left( 1-s\right) ds\geq \int_{0}^{1}f^{\prime \prime }\left(
\left( 1-s\right) t+sz\right) \left( 1-s\right) ds\geq k\int_{0}^{1}\left(
1-s\right) ds,
\end{equation*}%
namely%
\begin{equation*}
\frac{1}{2}K\geq \int_{0}^{1}f^{\prime \prime }\left( \left( 1-s\right)
t+sz\right) \left( 1-s\right) ds\geq \frac{1}{2}k,
\end{equation*}%
and by (\ref{e.3.5}) we get the double inequality%
\begin{equation}
\frac{1}{2}K\left( z-t\right) ^{2}\geq f\left( z\right) -f\left( t\right)
-\left( z-t\right) f^{\prime }\left( t\right) \geq \frac{1}{2}k\left(
z-t\right) ^{2}  \label{e.3.6}
\end{equation}%
for any $z,$ $t\in \mathring{I}.$

Fix $t\in \left[ m,M\right] $. Using Lemma \ref{l.1.2} and the inequality (%
\ref{e.3.6}) we obtain for the element $c\in A$ with $\sigma \left( c\right)
\subseteq \left[ m,M\right] \subset I$ the following inequality in the order
of $A$ 
\begin{equation*}
\frac{1}{2}K\left( c-t\right) ^{2}\geq f\left( c\right) -f\left( t\right)
-\left( c-t\right) f^{\prime }\left( t\right) \geq \frac{1}{2}k\left(
c-t\right) ^{2}.
\end{equation*}%
If we take in this inequality the functional $\psi $ we get (\ref{e.3.1}).

Fix $z\in \left[ m,M\right] .$ Using Lemma \ref{l.1.2} and the inequality (%
\ref{e.3.6}) we obtain for the element $c\in A$ with $\sigma \left( c\right)
\subseteq \left[ m,M\right] \subset I$ the following inequality in the order
of $A$%
\begin{equation}
\frac{1}{2}K\left( c-z\right) ^{2}\geq f\left( z\right) -f\left( c\right)
-zf^{\prime }\left( c\right) +cf^{\prime }\left( c\right) \geq \frac{1}{2}%
k\left( c-z\right) ^{2}.  \label{e.3.7}
\end{equation}%
If we take in this inequality the functional $\psi $ we get%
\begin{align*}
\frac{1}{2}K\psi \left[ \left( c-z\right) ^{2}\right] & \geq \psi \left(
cf^{\prime }\left( c\right) \right) -z\psi \left( f^{\prime }\left( c\right)
\right) -\psi \left( f\left( c\right) \right) +f\left( z\right) \\
& \geq \frac{1}{2}k\psi \left[ \left( c-z\right) ^{2}\right] ,
\end{align*}%
for any $z\in \left[ m,M\right] .$ If we replace $z$ with $t$ we get the
desired result (\ref{e.3.2}).
\end{proof}

\begin{corollary}
\label{c.3.1}With the assumptions of Theorem \ref{t.3.1} we have the
Jensen's type inequalities%
\begin{equation}
\frac{1}{2}K\left[ \psi \left( c^{2}\right) -\psi ^{2}\left( c\right) \right]
\geq \psi \left( f\left( c\right) \right) -f\left( \psi \left( c\right)
\right) \geq \frac{1}{2}k\left[ \psi \left( c^{2}\right) -\psi ^{2}\left(
c\right) \right]   \label{e.3.8}
\end{equation}%
and 
\begin{align}
\frac{1}{2}K\left[ \psi \left( c^{2}\right) -\psi ^{2}\left( c\right) \right]
& \geq \psi \left( cf^{\prime }\left( c\right) \right) -\psi \left( c\right)
\psi \left( f^{\prime }\left( c\right) \right) +f\left( \psi \left( c\right)
\right) -\psi \left( f\left( c\right) \right)   \label{e.3.9} \\
& \geq \frac{1}{2}k\left[ \psi \left( c^{2}\right) -\psi ^{2}\left( c\right) %
\right] .  \notag
\end{align}
\end{corollary}

Follows by Theorem \ref{t.3.1} on choosing $t=\psi \left( c\right) \in \left[
m,M\right] .$

\begin{corollary}
\label{c.3.2}With the assumptions of Theorem \ref{t.3.1} we have%
\begin{align}
& \frac{1}{2}K\psi \left[ \left( c-\frac{m+M}{2}\right) ^{2}\right] 
\label{e.3.10} \\
& \geq \psi \left( f\left( c\right) \right) -f^{\prime }\left( \frac{m+M}{2}%
\right) \left( \psi \left( c\right) -\frac{m+M}{2}\right) -f\left( \frac{m+M%
}{2}\right)   \notag \\
& \geq \frac{1}{2}k\psi \left[ \left( c-\frac{m+M}{2}\right) ^{2}\right]  
\notag
\end{align}%
and 
\begin{align}
& \frac{1}{2}K\psi \left[ \left( c-\frac{m+M}{2}\right) ^{2}\right] 
\label{e.3.11} \\
& \geq \psi \left( cf^{\prime }\left( c\right) \right) -\frac{m+M}{2}\psi
\left( f^{\prime }\left( c\right) \right) +f\left( \frac{m+M}{2}\right)
-\psi \left( f\left( c\right) \right)   \notag \\
& \geq \frac{1}{2}k\psi \left[ \left( c-\frac{m+M}{2}\right) ^{2}\right] . 
\notag
\end{align}
\end{corollary}

Follows by Theorem \ref{t.3.1} on choosing $t=\frac{m+M}{2}.$

\begin{corollary}
\label{c.3.3}With the assumptions of Theorem \ref{t.3.1} and, if, in
addition, $t=\frac{\psi \left( cf^{\prime }\left( c\right) \right) }{\psi
\left( f^{\prime }\left( c\right) \right) }\in \left[ m,M\right] $ with $%
\psi \left( f^{\prime }\left( c\right) \right) \neq 0,$ then we have the
Slater's type inequalities 
\begin{align}
\frac{1}{2}K\psi \left[ \left( c-\frac{\psi \left( cf^{\prime }\left(
c\right) \right) }{\psi \left( f^{\prime }\left( c\right) \right) }\right)
^{2}\right] & \geq f\left( \frac{\psi \left( cf^{\prime }\left( c\right)
\right) }{\psi \left( f^{\prime }\left( c\right) \right) }\right) -\psi
\left( f\left( c\right) \right)  \label{e.3.12} \\
& \geq \frac{1}{2}k\psi \left[ \left( c-\frac{\psi \left( cf^{\prime }\left(
c\right) \right) }{\psi \left( f^{\prime }\left( c\right) \right) }\right)
^{2}\right] ,  \notag
\end{align}%
and%
\begin{align}
& \frac{1}{2}K\psi \left[ \left( c-\frac{\psi \left( cf^{\prime }\left(
c\right) \right) }{\psi \left( f^{\prime }\left( c\right) \right) }\right)
^{2}\right]  \label{e.3.13} \\
& \geq f^{\prime }\left( \frac{\psi \left( cf^{\prime }\left( c\right)
\right) }{\psi \left( f^{\prime }\left( c\right) \right) }\right) \frac{\psi
\left( cf^{\prime }\left( c\right) \right) }{\psi \left( f^{\prime }\left(
c\right) \right) }-\psi \left( c\right) f^{\prime }\left( \frac{\psi \left(
cf^{\prime }\left( c\right) \right) }{\psi \left( f^{\prime }\left( c\right)
\right) }\right)  \notag \\
& -f\left( \frac{\psi \left( cf^{\prime }\left( c\right) \right) }{\psi
\left( f^{\prime }\left( c\right) \right) }\right) +\psi \left( f\left(
c\right) \right)  \notag \\
& \geq \frac{1}{2}k\psi \left[ \left( c-\frac{\psi \left( cf^{\prime }\left(
c\right) \right) }{\psi \left( f^{\prime }\left( c\right) \right) }\right)
^{2}\right] .  \notag
\end{align}
\end{corollary}

Follows by Follows by Theorem \ref{t.3.1} on choosing $t=\frac{\psi \left(
cf^{\prime }\left( c\right) \right) }{\psi \left( f^{\prime }\left( c\right)
\right) }\in \left[ m,M\right] .$ We observe that a sufficient condition for
this to happen is that $f^{\prime }\left( c\right) >0$ and $\psi \left(
f^{\prime }\left( c\right) \right) >0.$

\begin{corollary}
\label{c.3.4}With the assumptions of Theorem \ref{t.3.1} we have%
\begin{align}
& \frac{1}{4}K\left[ \frac{1}{12}\left( M-m\right) ^{2}+\psi \left[ \left( c-%
\frac{m+M}{2}\right) ^{2}\right] \right]  \label{e.3.14} \\
& \geq \frac{1}{2}\left[ \psi \left( f\left( c\right) \right) +\frac{\left(
M-\psi \left( c\right) \right) f\left( M\right) +\left( \psi \left( c\right)
-m\right) f\left( m\right) }{M-m}\right]  \notag \\
& -\frac{1}{M-m}\int_{m}^{M}f\left( t\right) dt  \notag \\
& \geq \frac{1}{4}k\left[ \frac{1}{12}\left( M-m\right) ^{2}+\psi \left[
\left( c-\frac{m+M}{2}\right) ^{2}\right] \right]  \notag
\end{align}%
and%
\begin{align}
& \frac{1}{2}K\left[ \frac{1}{12}\left( M-m\right) ^{2}+\psi \left[ \left( c-%
\frac{m+M}{2}\right) ^{2}\right] \right]  \label{e.3.15} \\
& \geq \frac{1}{M-m}\int_{m}^{M}f\left( z\right) dz-\psi \left( f\left(
c\right) \right) -\frac{m+M}{2}\psi \left( f^{\prime }\left( c\right)
\right) -\psi \left( cf^{\prime }\left( c\right) \right)  \notag \\
& \geq \frac{1}{2}k\left[ \frac{1}{12}\left( M-m\right) ^{2}+\psi \left[
\left( c-\frac{m+M}{2}\right) ^{2}\right] \right] .  \notag
\end{align}
\end{corollary}

\begin{proof}
If we take the integral mean over $t$ on $\left[ m,M\right] $ in the
inequality (\ref{e.3.6}) we get%
\begin{align}
& \frac{1}{2}K\frac{1}{M-m}\int_{m}^{M}\left( z-t\right) ^{2}dt
\label{e.3.16} \\
& \geq f\left( z\right) -\frac{1}{M-m}\int_{m}^{M}f\left( t\right) dt-\frac{1%
}{M-m}\int_{m}^{M}\left( z-t\right) f^{\prime }\left( t\right) dt  \notag \\
& \geq \frac{1}{2}\frac{1}{M-m}\int_{m}^{M}\left( z-t\right) ^{2}dt  \notag
\end{align}%
for any $z\in \left[ m,M\right] .$

Observe that%
\begin{align*}
\frac{1}{M-m}\int_{m}^{M}\left( z-t\right) ^{2}& =\frac{\left( M-z\right)
^{3}+\left( z-m\right) ^{3}}{3\left( M-m\right) } \\
& =\frac{1}{3}\left[ \left( z-m\right) ^{2}+\left( M-z\right) ^{2}-\left(
z-m\right) \left( M-z\right) \right] \\
& =\frac{1}{3}\left[ \frac{1}{4}\left( M-m\right) ^{2}+3\left( z-\frac{m+M}{2%
}\right) ^{2}\right] \\
& =\frac{1}{12}\left( M-m\right) ^{2}+\left( z-\frac{m+M}{2}\right) ^{2}
\end{align*}%
and%
\begin{align*}
& \frac{1}{M-m}\int_{m}^{M}\left( z-t\right) f^{\prime }\left( t\right) dt \\
& =\frac{1}{M-m}\left[ \left. \left( z-t\right) f\left( t\right) \right\vert
_{m}^{M}+\int_{m}^{M}f\left( t\right) dt\right] \\
& =\frac{1}{M-m}\left[ \int_{m}^{M}f\left( t\right) dt-\left( M-z\right)
f\left( M\right) -\left( z-m\right) f\left( m\right) \right] \\
& =\frac{1}{M-m}\int_{m}^{M}f\left( t\right) dt-\frac{\left( M-z\right)
f\left( M\right) +\left( z-m\right) f\left( m\right) }{M-m}.
\end{align*}%
Then by (\ref{e.3.16}) we get%
\begin{align*}
& \frac{1}{2}K\left[ \frac{1}{12}\left( M-m\right) ^{2}+\left( z-\frac{m+M}{2%
}\right) ^{2}\right] \\
& \geq f\left( z\right) -\frac{1}{M-m}\int_{m}^{M}f\left( t\right) dt-\frac{1%
}{M-m}\int_{m}^{M}f\left( t\right) dt \\
& +\frac{\left( M-z\right) f\left( M\right) +\left( z-m\right) f\left(
m\right) }{M-m} \\
& \geq \frac{1}{2}k\left[ \frac{1}{12}\left( M-m\right) ^{2}+\left( z-\frac{%
m+M}{2}\right) ^{2}\right]
\end{align*}%
namely%
\begin{align}
& \frac{1}{4}K\left[ \frac{1}{12}\left( M-m\right) ^{2}+\left( z-\frac{m+M}{2%
}\right) ^{2}\right]  \label{e.3.17} \\
& \geq \frac{1}{2}\left[ f\left( z\right) +\frac{\left( M-z\right) f\left(
M\right) +\left( z-m\right) f\left( m\right) }{M-m}\right] -\frac{1}{M-m}%
\int_{m}^{M}f\left( t\right) dt  \notag \\
& \geq \frac{1}{4}k\left[ \frac{1}{12}\left( M-m\right) ^{2}+\left( z-\frac{%
m+M}{2}\right) ^{2}\right]  \notag
\end{align}%
for any $z\in \left[ m,M\right] .$

Using Lemma \ref{l.1.2} and the inequality (\ref{e.3.17}) we obtain for the
element $c\in A$ with $\sigma \left( c\right) \subseteq \left[ m,M\right]
\subset I$ the following inequality in the order of $A$%
\begin{align*}
& \frac{1}{4}K\left[ \frac{1}{12}\left( M-m\right) ^{2}+\left( c-\frac{m+M}{2%
}\right) ^{2}\right] \\
& \geq \frac{1}{2}\left[ f\left( c\right) +\frac{\left( M-c\right) f\left(
M\right) +\left( c-m\right) f\left( m\right) }{M-m}\right] -\frac{1}{M-m}%
\int_{m}^{M}f\left( t\right) dt \\
& \geq \frac{1}{4}k\left[ \frac{1}{12}\left( M-m\right) ^{2}+\left( c-\frac{%
m+M}{2}\right) ^{2}\right] .
\end{align*}%
If we apply to this inequality the functional $\psi $ we get (\ref{e.3.14}).

If we take the integral mean over $z$ on $\left[ m,M\right] $ in the
inequality (\ref{e.3.6}) we get%
\begin{align*}
& \frac{1}{2}K\frac{1}{M-m}\int_{m}^{M}\left( z-t\right) ^{2}dz \\
& \geq \frac{1}{M-m}\int_{m}^{M}f\left( z\right) dz-f\left( t\right) -\left( 
\frac{m+M}{2}-t\right) f^{\prime }\left( t\right) \\
& \geq \frac{1}{2}k\frac{1}{M-m}\int_{m}^{M}\left( z-t\right) ^{2}dz,
\end{align*}%
namely%
\begin{align}
& \frac{1}{2}K\left[ \frac{1}{12}\left( M-m\right) ^{2}+\left( t-\frac{m+M}{2%
}\right) ^{2}\right]  \label{e.3.18} \\
& \geq \frac{1}{M-m}\int_{m}^{M}f\left( z\right) dz-f\left( t\right) -\left( 
\frac{m+M}{2}-t\right) f^{\prime }\left( t\right)  \notag \\
& \geq \frac{1}{2}k\left[ \frac{1}{12}\left( M-m\right) ^{2}+\left( t-\frac{%
m+M}{2}\right) ^{2}\right]  \notag
\end{align}%
for any $t\in \left[ m,M\right] .$

Using (\ref{e.3.18}) and a similar argument as above, we get the desired
result (\ref{e.3.15}).
\end{proof}

\section{Some Examples}

Assume that $A$ is a Hermitian unital Banach $\ast $-algebra and $\psi
:A\rightarrow \mathbb{C}$ a positive normalized linear functional on $A.$

Let $c\in A$ be a selfadjoint element with $\sigma \left( c\right) \subseteq %
\left[ m,M\right] $ for some real numbers $m<M.$ If we take $f\left(
t\right) =t^{2}$ and calculate%
\begin{equation*}
\Theta _{f}\left( t;m,M\right) =\frac{M^{2}-t^{2}}{M-t}-\frac{t^{2}-m^{2}}{%
t-m}=M-m
\end{equation*}%
then by (\ref{e.2.3}) we get%
\begin{equation}
0\leq \psi \left( c^{2}\right) -\left( \psi \left( c\right) \right) ^{2}\leq
\left( M-\psi \left( c\right) \right) \left( \psi \left( c\right) -m\right)
\leq \frac{1}{4}\left( M-m\right) ^{2}.  \label{e.4.1}
\end{equation}

Consider the convex function $f:\left[ m,M\right] \subset \left( 0,\infty
\right) \rightarrow \left( 0,\infty \right) ,$ $f\left( t\right) =t^{p},$ $%
p>1.$ Using the inequality (\ref{e.2.3}) we have 
\begin{align}
0& \leq \psi \left( c^{p}\right) -\left( \psi \left( c\right) \right)
^{p}\leq p\left( M-\psi \left( c\right) \right) \left( \psi \left( c\right)
-m\right) \frac{M^{p-1}-m^{p-1}}{M-m}  \label{e.4.2} \\
& \leq \frac{1}{4}p\left( M-m\right) \left( M^{p-1}-m^{p-1}\right)  \notag
\end{align}
for any $c\in A$ a selfadjoint element with $\sigma \left( c\right)
\subseteq \left[ m,M\right] \subset \left( 0,\infty \right) .$

If we use the inequality (\ref{e.2.9}) we also get 
\begin{align}
0& \leq \psi \left( c^{p}\right) -\left( \psi \left( c\right) \right) ^{p}
\label{e.4.3} \\
& \leq \left( 1+2\frac{\left\vert \psi \left( c\right) -\frac{m+M}{2}%
\right\vert }{M-m}\right) \left[ \frac{m^{p}+M^{p}}{2}-\left( \frac{m+M}{2}%
\right) ^{p}\right]   \notag \\
& \leq m^{p}+M^{p}-2^{1-p}\left( m+M\right) ^{p}  \notag
\end{align}%
for any $c\in A$ a selfadjoint element with $\sigma \left( c\right)
\subseteq \left[ m,M\right] \subset \left( 0,\infty \right) .$

Since $f^{\prime \prime }\left( t\right) =p\left( p-1\right) t^{p-2},$ $t>0$
then 
\begin{align}
k_{p}& :=p\left( p-1\right) \left\{ 
\begin{array}{c}
M^{p-2}\text{ for }p\in \left( 1,2\right)  \\ 
\\ 
m^{p-2}\text{ for }p\in \lbrack 2,\infty )%
\end{array}%
\right.   \label{e.4.4} \\
&  \notag \\
& \leq f^{\prime \prime }\left( t\right) \leq K_{p}:=p\left( p-1\right)
\left\{ 
\begin{array}{c}
m^{p-2}\text{ for }p\in \left( 1,2\right)  \\ 
\\ 
M^{p-2}\text{ for }p\in \lbrack 2,\infty )%
\end{array}%
\right.   \notag
\end{align}%
for any $t\in \left[ m,M\right] .$

Using (\ref{e.3.8}) and (\ref{e.3.9}) we get%
\begin{equation}
\frac{1}{2}K_{p}\left[ \psi \left( c^{2}\right) -\psi ^{2}\left( c\right) %
\right] \geq \psi \left( c^{p}\right) -\left( \psi \left( c\right) \right)
^{p}\geq \frac{1}{2}k_{p}\left[ \psi \left( c^{2}\right) -\psi ^{2}\left(
c\right) \right]  \label{e.4.5}
\end{equation}%
and 
\begin{align}
\frac{1}{2}K_{p}\left[ \psi \left( c^{2}\right) -\psi ^{2}\left( c\right) %
\right] & \geq \left( p-1\right) \psi \left( c^{p}\right) +\left( \psi
\left( c\right) \right) ^{p}-p\psi \left( c\right) \psi \left( c^{p-1}\right)
\label{e.4.6} \\
& \geq \frac{1}{2}k_{p}\left[ \psi \left( c^{2}\right) -\psi ^{2}\left(
c\right) \right] ,  \notag
\end{align}%
for any $c\in A$ a selfadjoint element with $\sigma \left( c\right)
\subseteq \left[ m,M\right] \subset \left( 0,\infty \right) .$

Using (\ref{e.3.12}) and (\ref{e.3.13}) we get 
\begin{align}
\frac{1}{2}K_{p}\psi \left[ \left( c-\frac{\psi \left( c^{p}\right) }{\psi
\left( c^{p-1}\right) }\right) ^{2}\right] & \geq \left( \frac{\psi \left(
c^{p}\right) }{\psi \left( c^{p-1}\right) }\right) ^{p}-\psi \left(
c^{p}\right)  \label{e.4.7} \\
& \geq \frac{1}{2}k_{p}\psi \left[ \left( c-\frac{\psi \left( c^{p}\right) }{%
\psi \left( c^{p-1}\right) }\right) ^{2}\right] ,  \notag
\end{align}%
and%
\begin{align}
& \frac{1}{2}K_{p}\psi \left[ \left( c-\frac{\psi \left( c^{p}\right) }{\psi
\left( c^{p-1}\right) }\right) ^{2}\right]  \label{e.4.8} \\
& \geq p\left( \frac{\psi \left( c^{p}\right) }{\psi \left( c^{p-1}\right) }%
\right) ^{p-1}\left( \frac{\psi \left( c^{p}\right) }{\psi \left(
c^{p-1}\right) }-\psi \left( c\right) \right) -\left( \frac{\psi \left(
c^{p}\right) }{\psi \left( c^{p-1}\right) }\right) ^{p}+\psi \left(
c^{p}\right)  \notag \\
& \geq \frac{1}{2}k_{p}\psi \left[ \left( c-\frac{\psi \left( c^{p}\right) }{%
\psi \left( c^{p-1}\right) }\right) ^{2}\right]  \notag
\end{align}%
for any $c\in A$ a selfadjoint element with $\sigma \left( c\right)
\subseteq \left[ m,M\right] \subset \left( 0,\infty \right) .$

Using (\ref{e.3.14}) and (\ref{e.3.15}) we also have%
\begin{align}
& \frac{1}{4}K\left[ \frac{1}{12}\left( M-m\right) ^{2}+\psi \left[ \left( c-%
\frac{m+M}{2}\right) ^{2}\right] \right]   \label{e.4.8.1} \\
& \geq \frac{1}{2}\left[ \psi \left( c^{p}\right) +\frac{\left( M-\psi
\left( c\right) \right) M^{p}+\left( \psi \left( c\right) -m\right) m^{p}}{%
M-m}\right]   \notag \\
& -\frac{M^{p+1}-m^{p+1}}{\left( p+1\right) \left( M-m\right) }  \notag \\
& \geq \frac{1}{4}k\left[ \frac{1}{12}\left( M-m\right) ^{2}+\psi \left[
\left( c-\frac{m+M}{2}\right) ^{2}\right] \right]   \notag
\end{align}%
and%
\begin{align}
& \frac{1}{2}K_{p}\left[ \frac{1}{12}\left( M-m\right) ^{2}+\psi \left[
\left( c-\frac{m+M}{2}\right) ^{2}\right] \right]   \label{e.4.9} \\
& \geq \frac{M^{p+1}-m^{p+1}}{\left( p+1\right) \left( M-m\right) }-p\frac{%
m+M}{2}\psi \left( c^{p-1}\right) -\left( p+1\right) \psi \left(
c^{p}\right)   \notag \\
& \geq \frac{1}{2}k_{p}\left[ \frac{1}{12}\left( M-m\right) ^{2}+\psi \left[
\left( c-\frac{m+M}{2}\right) ^{2}\right] \right]   \notag
\end{align}%
for any $c\in A$ a selfadjoint element with $\sigma \left( c\right)
\subseteq \left[ m,M\right] \subset \left( 0,\infty \right) .$

Consider the convex function $f:\left[ m,M\right] \subset \left( 0,\infty
\right) \rightarrow \left( 0,\infty \right) $, $f\left( t\right) =\frac{1}{t}
$. We have%
\begin{equation*}
\Theta _{f}\left( t;m,M\right) =\frac{\frac{1}{M}-\frac{1}{t}}{M-t}-\frac{%
\frac{1}{t}-\frac{1}{m}}{t-m}=\frac{M-m}{tmM},
\end{equation*}%
which implies that 
\begin{equation*}
\sup_{t\in \left( m,M\right) }\Theta _{f}\left( t;m,M\right) =\frac{M-m}{%
m^{2}M}.
\end{equation*}%
From (\ref{e.2.3}) we get%
\begin{align}
0& \leq \psi \left( c^{-1}\right) -\psi ^{-1}\left( c\right) \leq \frac{%
\left( M-\psi \left( c\right) \right) \left( \psi \left( c\right) -m\right) 
}{m^{2}M}  \label{e.4.10} \\
& \leq \left\{ 
\begin{array}{l}
\frac{1}{4m^{2}M}\left( M-m\right) ^{2} \\ 
\\ 
\left( M-\psi \left( c\right) \right) \left( \psi \left( c\right) -m\right) 
\frac{M+m}{m^{2}M^{2}}%
\end{array}%
\right. \leq \frac{1}{4}\left( M-m\right) ^{2}\frac{M+m}{M^{2}m^{2}}  \notag
\end{align}%
for any $c\in A$ a selfadjoint element with $\sigma \left( c\right)
\subseteq \left[ m,M\right] \subset \left( 0,\infty \right) .$

From (\ref{e.2.3.1}) we have 
\begin{equation}
0\leq \psi \left( c^{-1}\right) -\psi ^{-1}\left( c\right) \leq \frac{1}{4}%
\frac{\left( M-m\right) ^{2}}{mM}\psi ^{-1}\left( c\right) \leq \frac{1}{%
4m^{2}M}\left( M-m\right) ^{2}  \label{e.4.11}
\end{equation}%
for any $c\in A$ a selfadjoint element with $\sigma \left( c\right)
\subseteq \left[ m,M\right] \subset \left( 0,\infty \right) .$

From (\ref{e.2.9}) we also have%
\begin{align}
0& \leq \psi \left( c^{-1}\right) -\psi ^{-1}\left( c\right) \leq \frac{%
\left( M-m\right) ^{2}}{2mM\left( m+M\right) }\left( 1+2\frac{\left\vert
\psi \left( c\right) -\frac{m+M}{2}\right\vert }{M-m}\right) 
\label{e.4.11.1} \\
& \leq \frac{\left( M-m\right) ^{2}}{mM\left( m+M\right) }  \notag
\end{align}%
for any $c\in A$ a selfadjoint element with $\sigma \left( c\right)
\subseteq \left[ m,M\right] \subset \left( 0,\infty \right) .$

Since $f^{\prime \prime }\left( t\right) =\frac{2}{t^{3}},$ $t>0,$ then $%
\frac{2}{m^{3}}\geq f^{\prime \prime }\left( t\right) \geq \frac{2}{M^{3}}$
and by (\ref{e.3.8}) and (\ref{e.3.9}) we get 
\begin{equation}
\frac{1}{m^{3}}\left[ \psi \left( c^{2}\right) -\psi ^{2}\left( c\right) %
\right] \geq \psi \left( c^{-1}\right) -\psi ^{-1}\left( c\right) \geq \frac{%
1}{M^{3}}\left[ \psi \left( c^{2}\right) -\psi ^{2}\left( c\right) \right] 
\label{e.4.12}
\end{equation}%
and 
\begin{align}
\frac{1}{m^{3}}\left[ \psi \left( c^{2}\right) -\psi ^{2}\left( c\right) %
\right] & \geq \frac{1}{2}\left[ \psi \left( c\right) \psi \left(
c^{-2}\right) +\psi ^{-1}\left( c\right) \right] -\psi \left( c^{-1}\right) 
\label{e.4.13} \\
& \geq \frac{1}{M^{3}}\left[ \psi \left( c^{2}\right) -\psi ^{2}\left(
c\right) \right] ,  \notag
\end{align}%
for any $c\in A$ a selfadjoint element with $\sigma \left( c\right)
\subseteq \left[ m,M\right] \subset \left( 0,\infty \right) .$

From (\ref{e.3.12}) and (\ref{e.3.13}) we also have  
\begin{align}
\frac{1}{m^{3}}\psi \left[ \left( c-\frac{\psi \left( c^{-1}\right) }{\psi
\left( c^{-2}\right) }\right) ^{2}\right] & \geq \frac{\psi \left(
c^{-2}\right) }{\psi \left( c^{-1}\right) }-\psi \left( c^{-1}\right) 
\label{e.4.14} \\
& \geq \frac{1}{M^{3}}\psi \left[ \left( c-\frac{\psi \left( c^{-1}\right) }{%
\psi \left( c^{-2}\right) }\right) ^{2}\right] ,  \notag
\end{align}%
and%
\begin{align}
& \frac{1}{m^{3}}\psi \left[ \left( c-\frac{\psi \left( c^{-1}\right) }{\psi
\left( c^{-2}\right) }\right) ^{2}\right]   \label{e.4.15} \\
& \geq \psi \left( c^{-1}\right) -\frac{\psi \left( c^{-1}\right) }{\psi
\left( c^{-2}\right) }+\psi \left( c\right) \frac{\psi ^{2}\left(
c^{-2}\right) }{\psi ^{2}\left( c^{-1}\right) }-\frac{\psi \left(
c^{-2}\right) }{\psi \left( c^{-1}\right) }  \notag \\
& \geq \frac{1}{M^{3}}\psi \left[ \left( c-\frac{\psi \left( c^{-1}\right) }{%
\psi \left( c^{-2}\right) }\right) ^{2}\right]   \notag
\end{align}%
for any $c\in A$ a selfadjoint element with $\sigma \left( c\right)
\subseteq \left[ m,M\right] \subset \left( 0,\infty \right) .$

Similar results may be stated for the convex functions $f\left( t\right)
=t^{r},$ $r<0$ and $f\left( t\right) =-t^{q}$, $q\in \left( 0,1\right) .$

The case of logarithmic function is also of interest. If we take the
function $f\left( t\right) =-\ln t$ in (\ref{e.2.3}), then we get%
\begin{equation}
0\leq \ln \left( \psi \left( c\right) \right) -\psi \left( \ln c\right) \leq 
\frac{\left( M-\psi \left( c\right) \right) \left( \psi \left( c\right)
-m\right) }{mM}\leq \frac{1}{4}\frac{\left( M-m\right) ^{2}}{mM}
\label{e.4.16}
\end{equation}%
for any $c\in A$ a selfadjoint element with $\sigma \left( c\right)
\subseteq \left[ m,M\right] \subset \left( 0,\infty \right) .$

From (\ref{e.2.9}) we have 
\begin{align}
0& \leq \ln \left( \psi \left( c\right) \right) -\psi \left( \ln c\right)
\leq \ln \left( \frac{m+M}{2\sqrt{mM}}\right) \left( 1+2\frac{\left\vert
\psi \left( c\right) -\frac{m+M}{2}\right\vert }{M-m}\right)   \label{e.4.17}
\\
& \leq \ln \left( \frac{m+M}{2\sqrt{mM}}\right) ^{2}  \notag
\end{align}%
for any $c\in A$ a selfadjoint element with $\sigma \left( c\right)
\subseteq \left[ m,M\right] \subset \left( 0,\infty \right) .$

Since $f^{\prime \prime }\left( t\right) =\frac{1}{t^{2}}$\ and $\frac{1}{%
m^{2}}\geq f^{\prime \prime }\left( t\right) \geq \frac{1}{M^{2}}$\ for any $%
t\in \left[ m,M\right] \subset \left( 0,\infty \right) ,$ then by (\ref%
{e.3.8}) and (\ref{e.3.9}) we have%
\begin{equation}
\frac{1}{2m^{2}}\left[ \psi \left( c^{2}\right) -\psi ^{2}\left( c\right) %
\right] \geq \ln \left( \psi \left( c\right) \right) -\psi \left( \ln
c\right) \geq \frac{1}{2M^{2}}\left[ \psi \left( c^{2}\right) -\psi
^{2}\left( c\right) \right]   \label{e.4.18}
\end{equation}%
and 
\begin{align}
\frac{1}{2m^{2}}\left[ \psi \left( c^{2}\right) -\psi ^{2}\left( c\right) %
\right] & \geq \psi \left( \ln c\right) -\ln \left( \psi \left( c\right)
\right) +\psi \left( c\right) \psi \left( c^{-1}\right) -1  \label{e.4.19} \\
& \geq \frac{1}{2M^{2}}\left[ \psi \left( c^{2}\right) -\psi ^{2}\left(
c\right) \right] ,  \notag
\end{align}%
for any $c\in A$ a selfadjoint element with $\sigma \left( c\right)
\subseteq \left[ m,M\right] \subset \left( 0,\infty \right) .$

Finally, by making use of (\ref{e.3.12}) and (\ref{e.3.13}) we have 
\begin{align}
\frac{1}{2m^{2}}\psi \left[ \left( c-\psi ^{-1}\left( c^{-1}\right) \right)
^{2}\right] & \geq \psi \left( \ln c\right) -\ln \left( \psi ^{-1}\left(
c^{-1}\right) \right)   \label{e.4.20} \\
& \geq \frac{1}{2M^{2}}\psi \left[ \left( c-\psi ^{-1}\left( c^{-1}\right)
\right) ^{2}\right] ,  \notag
\end{align}%
and%
\begin{align}
\frac{1}{2m^{2}}\psi \left[ \left( c-\psi ^{-1}\left( c^{-1}\right) \right)
^{2}\right] & \geq \psi \left( c\right) \psi \left( c^{-1}\right) -1-\psi
\left( \ln c\right) +\ln \left( \psi ^{-1}\left( c^{-1}\right) \right) 
\label{e.4.21} \\
& \geq \frac{1}{2M^{2}}\psi \left[ \left( c-\psi ^{-1}\left( c^{-1}\right)
\right) ^{2}\right]   \notag
\end{align}%
for any $c\in A$ a selfadjoint element with $\sigma \left( c\right)
\subseteq \left[ m,M\right] \subset \left( 0,\infty \right) .$

The interested reader may obtain other similar inequalities by using the
convex functions $f\left( t\right) =t\ln t,$ $t>0$ and $f\left( t\right)
=\exp \left( \alpha t\right) ,$ $t,$ $\alpha \in \mathbb{R}$ and $\alpha
\neq 0.$

\end{document}